\documentclass[11pt,reqno]{amsart}
\usepackage[tmargin=1in,bmargin=1in,rmargin=1in,lmargin=1in]{geometry}

\usepackage[breaklinks=true]{hyperref}

\usepackage{amsmath,amsthm,amsfonts,amssymb}

\theoremstyle{plain}
\newtheorem{theorem}{Theorem}[section]

\newtheorem{lemma}[theorem]{Lemma}

\newtheorem{cor}[theorem]{Corollary}
\newtheorem{utheorem}{\textrm{\textbf{Theorem}}}

\newcommand{\bd}[1]{d^{\langle #1 \rangle}}
\newcommand{\z}[1]{z^{\langle #1 \rangle}}

\theoremstyle{definition}
\newtheorem{defn}[theorem]{Definition}

\newtheorem{rem}[theorem]{Remark}

\numberwithin{equation}{section}

\DeclareMathOperator{\adj}{adj}

\newcommand{\altr}[1]{\mathbb{R}^{#1}_{\rm alt}}

\begin{document}
\title[Sign non-reversal property for TN/TP matrices, and TP tests for
interval hulls]{Sign non-reversal property for totally non-negative and
totally positive matrices, and testing total positivity of their interval
hull}

\author{Projesh Nath Choudhury}
\address[P.N.~Choudhury]{Department of Mathematics, Indian Institute of
Science, Bangalore 560012, India}
\email{\tt projeshc@iisc.ac.in, projeshnc@alumni.iitm.ac.in}

\author{M. Rajesh Kannan}
\address[M.R.~Kannan]{Department of Mathematics, Indian Institute of
Technology Kharagpur, Kharagpur 721302, India}
\email{\tt rajeshkannan@maths.iitkgp.ac.in, rajeshkannan1.m@gmail.com}

\author{Apoorva Khare}
\address[A.~Khare]{Department of Mathematics, Indian Institute of
Science, Bangalore 560012, India; and Analysis \& Probability Research
Group, Bangalore 560012, India}
\email{\tt khare@iisc.ac.in}

\date{\today}

\begin{abstract}
A matrix $A$ is totally positive (or non-negative) of order $k$, denoted
$TP_k$ (or $TN_k$), if all minors of size $\leq k$ are positive (or
non-negative). It is well-known that such matrices are characterized by
the variation diminishing property together with the sign non-reversal
property. We do away with the former, and show that $A$ is $TP_k$ if and
only if every submatrix formed from at most $k$ consecutive
rows and columns has the sign non-reversal property. In fact this can be
strengthened to only consider test vectors in $\mathbb{R}^k$ with
alternating signs. We also show a similar characterization for all $TN_k$
matrices -- more strongly, both of these characterizations use a single
vector (with alternating signs) for each square submatrix. These
characterizations are novel, and similar in spirit to the fundamental
results characterizing $TP$ matrices by Gantmacher--Krein
[\textit{Compos.\ Math.} 1937] and $P$-matrices by Gale--Nikaido
[\textit{Math.\ Ann.} 1965].
    
As an application, we study the interval hull $\mathbb{I}(A,B)$ of two $m
\times n$ matrices $A=(a_{ij})$ and $B = (b_{ij})$. This is the
collection of $C \in \mathbb{R}^{m \times n}$ such that each $c_{ij}$ is
between $a_{ij}$ and $b_{ij}$. Using the sign non-reversal property, we
identify a two-element subset of $\mathbb{I}(A,B)$ that detects the
$TP_k$ property for all of $\mathbb{I}(A,B)$ for arbitrary $k \geq 1$. In
particular, this provides a test for total positivity (of any order),
simultaneously for an entire class of rectangular matrices.
In parallel, we also provide a finite set to test the total
non-negativity (of any order) of an interval hull $\mathbb{I}(A,B)$.
\end{abstract}



\maketitle

\section{Introduction and main results}

Given an integer $k\geq 1$, a matrix is \textit{totally positive of order
$k$ ($TP_k$)} if all its minors of order at most $k$ are positive, and
\textit{totally positive (TP)} if all its minors are positive. Similarly,
one defines \textit{totally non-negative (TN)} and $TN_k$ matrices for $k
\geq 1$. These classes of matrices have important applications in diverse
areas in mathematics, including analysis, approximation theory, cluster
algebras, combinatorics, differential equations, Gabor analysis,
integrable systems, matrix theory, probability and statistics, and
representation theory
\cite{BFZ96,Bre95,fallat-john,FZ02,GRS18,K68,KW14,Lu94,pinkus,Ri03,Sch07}.

A property intimately linked with total positivity is variation
diminution, which may be regarded as originating in the famous 1883
memoir of Laguerre~\cite{Laguerre}. Laguerre, following up on Descartes'
rule of signs~\cite{Descartes}, presented numerous results on the sign
changes in the coefficients of power series -- which he termed
`variations'. One such result says that \textit{if $f(x)$ is a polynomial
and $s \geq 0$, then the number $var(e^{sx} f(x))$ of variations in the
Maclaurin coefficients of $e^{sx} f(x)$ do not increase with $s$, hence
are bounded above by $var(f) < \infty$.} In 1912, in correspondence with
P\'olya~\cite{FP12}, Fekete reformulated and proved this result using
(what are known today as) one-sided P\'olya frequency sequences and their
variation diminishing property. This led P\'olya to coin the phrase
`variation diminishing' (or `variationsvermindernd' in German) in the
matrix-theoretic setting. We recall some of the fundamental contributions
in this setting: in 1930, Schoenberg~\cite{S30} showed that $TN$ matrices
(in fact sign-regular matrices) satisfy the variation diminishing
property. The complete characterization of this property was achieved in
1936 by Motzkin in his thesis~\cite{Mot36}. Subsequently, in their 1937
paper~\cite{gantmacher-krein}, Gantmacher--Krein showed that $TP_k$ and
$TN_k$ matrices are characterized by the positivity (or non-negativity)
of the spectra of all submatrices of size $\leq k$ (see
Theorem~\ref{Tgk}). In their 1950 book~\cite{GK50}, Gantmacher--Krein
made further fundamental contributions to total positivity and variation
diminution. In particular, they characterized $TN$ as follows (cited from
a later source):

\begin{theorem}[{\cite[Theorem 3.4]{pinkus}}]\label{Tgk-vardim}
Given a real $m \times n$ matrix $A$, the following statements are
equivalent.
\begin{enumerate}
	\item $A$ is totally non-negative.
	\item For all $x \in \mathbb{R}^n$, $S^-(Ax) \leq S^-(x)$. If
	moreover equality occurs and $Ax \neq 0$, the first (last)
	nonzero component of $Ax$ has the same sign as the first (last)
	nonzero component of $x$. Here $S^{-}(x)$ denotes the number of
	changes in sign after deleting all zero entries in $x$.
\end{enumerate}
\end{theorem}

The first and second sentences in assertion~(2) are known as the
`variation diminishing property' and the `sign non-reversal property',
respectively. Thus, these properties together characterize totally
non-negative matrices. A similar result holds for $TP$ matrices; see
e.g.~\cite[Theorem 3.3]{pinkus}.

In this short note, our goal is to show that the $TP_k$ and $TN_k$
properties are each equivalent to sign non-reversal alone. This provides
characterizations -- parallel to the above fundamental 20th-century
results -- that are equally simple, and remarkably, seem to our knowledge
(and that of experts) to be novel.

To state these results, we isolate the following definitions, used
below without further reference.

\begin{defn}
Let $n \geq 1$ be an integer, and $S \subseteq \mathbb{R}^n$ a subset.
\begin{enumerate}
\item Define the set $\langle n \rangle := \{ 1, \dots, n \}$ and the
vector $\bd{n} := (1, -1, \dots, (-1)^{n-1})^T \in \mathbb{R}^n$.

\item A matrix $A \in \mathbb{R}^{n \times n}$ has the \textit{sign
non-reversal property} with respect to $S$, if for all vectors $0 \neq x
\in S$, there is some coordinate $i \in \langle n \rangle$ such that $x_i
(Ax)_i > 0$.

\item We will also need a non-strict version. A matrix $A \in
\mathbb{R}^{n \times n}$ has the \textit{non-strict sign non-reversal
property} with respect to $S$ if for all vectors $0 \neq x \in S$, there
is some coordinate $i \in \langle n \rangle$ such that $x_i \neq 0$ and
$x_i (Ax)_i \geq 0$.

\item Let $\altr{n} \subset \mathbb{R}^n$ comprise the vectors with all
nonzero components and alternating signs.

\item Given $z = (z_1, \dots, z_n)^T \in \{ \pm 1 \}^n$, define $D_z$ to
be the diagonal matrix with $(i,i)$th entry $z_i$.

\item Finally, given two matrices $A,B \in \mathbb{R}^{m \times n}$, and
tuples of signs $z \in \{ \pm 1 \}^m, z' \in \{ \pm 1 \}^n$, define the
$m \times n$ matrices $|A|$, $I_{z,z'}(A,B)$, and $C^\pm(A,B)$ via:
\[
|A|_{ij} := |a_{ij}|, \qquad I_{z,z'}(A,B) :=
\frac{A+B}{2} - D_z \frac{|A-B|}{2} D_{z'}, \qquad C^\pm(A,B) :=
I_{\bd{m},\; \pm \bd{n}}(A,B).
\]
\end{enumerate}
\end{defn}

Now our first main result characterizes total positivity in terms of
increasingly weaker statements involving sign non-reversal.
Here and below, we use the notion of a \textit{contiguous submatrix},
i.e.~one whose rows and columns are indexed by sets of consecutive
integers.

\begin{utheorem}\label{tp-sign-rev_k}
	Let $m,n \geq k \geq 1$ be integers.
	Given $A \in \mathbb{R}^{m \times n}$, the following statements
	are equivalent.
	\begin{enumerate}
	\item The matrix $A$ is totally positive of order $k$.
	\item Every square submatrix of $A$ of size $r \leq k$ has the
	sign non-reversal property with respect to $\mathbb{R}^r$.
	\item Every contiguous square submatrix of $A$ of size $ r \leq
	k$ has the sign non-reversal property with respect to
	$\altr{r}$.\smallskip

	In fact, this is equivalent to non-strict sign non-reversal at a
	single vector:\smallskip

	\item For every $r \in \langle k \rangle$ and contiguous $r
	\times r$ submatrix $B$ of $A$, define the vector
	\begin{equation}\label{Ezb}
	z^B := \det(B) \adj(B) \bd{r}
	\end{equation}
	where $\adj(B)$ is the adjugate matrix of $B$. Now:
	(i)~$Bx \neq 0$ for all $x \in \altr{r}$; and
	(ii)~$B$ has the non-strict sign non-reversal property
	with respect to $z^B$.
	\end{enumerate}
\end{utheorem}

Notice that~(4) is \textit{a priori} weaker than~(3).

\begin{rem}
A `coordinate-based' unpacking of the above characterization says:
\textit{A matrix $A \in \mathbb{R}^{m \times n}$ is $TP_k$ if and only if
the following holds:}

\textit{Suppose $y=\begin{pmatrix} 0_l\\x\\0_s \end{pmatrix}$, where $l
\geq 0, s \geq 0$, and $0<r \leq \min \{ m, n, k \}$ are integers such
that $l + r + s = n$, and $x \in \altr{r}$. Then for each $j \in \langle
m-r+1 \rangle$, there exists $i \in \langle r \rangle$ such that
$y_{l+i}(Ay)_{i+j-1} >0$.}
\end{rem}

To our knowledge (and that of experts), Theorem \ref{tp-sign-rev_k} is a
novel characterization of total positivity of a given order $k$ -- as
well as of $TN_k$, stated and proved below. We now provide an
application. The third assertion in the theorem helps to provide a test
for not just one matrix but an entire interval hull of matrices (also
termed `interval matrix') to be $TP_k$, by reducing it to two test
matrices. Given matrices $A, B \in \mathbb{R}^{m \times n}$, recall that
their interval hull, denoted by $\mathbb{I}(A,B)$, is defined as follows:
\begin{equation}
\mathbb{I}(A,B) = \{C \in \mathbb{R}^{m \times n}: c_{ij} = t_{ij} a_{ij}
+ (1 - t_{ij}) b_{ij}, t_{ij} \in [0,1]\} \label{hulleqn}.
\end{equation}

If $A\neq B$, their interval hull is an uncountable set. We say that
$\mathbb{I}(A,B)$ is $TP_k$ ($TN_k$) if every element in it is $TP_k$
($TN_k$). A natural question involves finding a minimal test set which
would determine if $\mathbb{I}(A,B)$ is $TP_k$. 
(See~\cite{GAT16} for a recent survey of interval matrix results,
including along these lines.) When
(a)~the interval consists of square matrices ($m=n$), and
(b)~the order of total positivity equals the dimension ($k=n$),
this was answered by Garloff~\cite{Gar82} in 1982. It is natural to ask
what happens when these two equality-constraints are not imposed. Again,
we could not find such a result in the literature. Our next result, an
application of Theorem~\ref{tp-sign-rev_k}, answers this question.

\begin{utheorem}\label{tp_hull_k}
	Let $m,n \geq k \geq 1$ be integers, and $A, B \in \mathbb{R}^{m
	\times n}$. Then all matrices in $\mathbb{I}(A, B)$ are $TP_k$
	if and only if the two matrices $C^\pm(A,B)$ are $TP_k$.
\end{utheorem}

\begin{rem}
Note that $C^\pm(A,B)$ are independent of $k$.
\end{rem}

\begin{rem}
It is natural to ask if `totally non-negative' analogues of
Theorems~\ref{tp-sign-rev_k} and~\ref{tp_hull_k} exist. We indeed provide
these below -- see Theorems~\ref{ThmC} and~\ref{Ttnk}.
\end{rem}

Our next -- and immediate -- application of Theorem~\ref{tp-sign-rev_k}
is a novel characterization of \textit{P\'olya frequency sequences of
order $k$}; recall these are real sequences $(c_n)_{n \in \mathbb{Z}}$
such that for all integers
\[
r \in \langle k \rangle, \qquad m_1 < \cdots < m_r, \qquad
n_1 < \cdots < n_r,
\]
the determinant $\det (c_{m_i - n_j})_{i,j=1}^r \geq 0$. If all such
determinants are in fact positive, we say the sequence is a
\textit{$TP_k$ P\'olya frequency sequence}.

\begin{cor}
Let $k \geq 1$ be an integer. A real sequence $(c_n)_{n \in \mathbb{Z}}$
is a $TP_k$ P\'olya frequency sequence, if and only if for all integers
$r \in \langle k \rangle$ and $l \in \mathbb{Z}$, and all $x \in
\altr{r}$, there exists $j_0 \in \langle r \rangle$ such that $x_{j_0}
\sum_{j=1}^r c_{l + j_0 - j} x_j > 0$.
\end{cor}

Indeed, this follows by applying Theorem~\ref{tp-sign-rev_k} to the
square submatrices
\[
\begin{pmatrix}
c_l & c_{l-1} & \cdots & c_{l-r+1}\\
c_{l+1} & c_l & \cdots & c_{l-r+2}\\
\vdots & \vdots & \ddots & \vdots\\
c_{l+r-1} & c_{l+r-2} & \cdots & c_l
\end{pmatrix}, \qquad l \in \mathbb{Z}.
\]

Our final results are the counterparts of Theorem~\ref{tp-sign-rev_k}
and~\ref{tp_hull_k} for $TN$ matrices, promised above. In contrast to
Theorem~\ref{tp-sign-rev_k}(4) for $TP_k$ matrices, the $TN_k$ property
turns out to be \textit{equivalent} to all (small enough) square
submatrices having the non-strict sign non-reversal property with respect
to a single, well-chosen vector -- which turns out to be either
alternating or zero:

\begin{utheorem}\label{ThmC}
Let $m,n \geq k \geq 1$ be integers. Given $A \in \mathbb{R}^{m \times
n}$, the following statements are equivalent.
\begin{enumerate}
\item The matrix $A$ is totally non-negative of order $k$.

\item Every square submatrix of $A$ of size $r \leq k$ has the non-strict
sign non-reversal property with respect to $\mathbb{R}^r$.

\item Every square submatrix of $A$ of size $r \leq k$ has the non-strict
sign non-reversal property with respect to $\altr{r}$.

\item For every $r \in \langle k \rangle$ and $r \times r$ submatrix $B$
of $A$, the matrix $B$ has the non-strict sign non-reversal property with
respect to the vector $z^B$, defined as
in~\eqref{Ezb}.
\end{enumerate}
\end{utheorem}

\begin{rem}
Theorems~\ref{tp-sign-rev_k} and~\ref{ThmC} are reminiscent of a
classical result of Ky Fan of a similar nature \cite[Theorem~5]{KyFan},
shown in the context of proving Ostrowski-type inequalities. We briefly
discuss it vis-a-vis the current results: firstly,
Theorems~\ref{tp-sign-rev_k} and~\ref{ThmC} require the use of only a
single vector $z^B$ as in~\eqref{Ezb}, in contrast to an uncountable test
set in~\cite{KyFan}. Next, Ky Fan studies $TP$ matrices; we are able to
account for both $TP$ and $TN$ matrices -- and moreover, we characterize
matrices that are $TP/TN$ of any order $k$. Finally, Ky Fan works with
all submatrices, whereas the above results deduce total positivity (of
order $k$) from working with just the contiguous submatrices (of size at
most $k$).
\end{rem}

Given Theorem~\ref{ThmC}, which is a $TN_k$ analogue of
Theorem~\ref{tp-sign-rev_k}, a natural question is to seek a similar
$TN_k$ analogue of Theorem~\ref{tp_hull_k}. Such a result was shown very
recently (2020) by Adm et al.~\cite{AG20}, for rectangular $TN$ matrices
$A,B \in \mathbb{R}^{m \times n}$. (See also the related
work~\cite{AG13}, which resolves a longstanding conjecture
from~\cite{Gar82} involving (square) nonsingular $TN$ interval matrices.)
Returning to~\cite{AG20}, the authors show that under certain technical
constraints, a minimal test set of two matrices suffices to check the
total non-negativity of the entire interval hull. Our final result
removes the technical assumptions in~\cite{AG20}, and holds for $TN_k$
interval hulls for arbitrary $k \geq 1$ -- at the cost of working with a
larger (but finite) test set:

\begin{utheorem}\label{Ttnk}
Let $m,n \geq k \geq 1$ be integers, and $A,B \in \mathbb{R}^{m \times
n}$. Then all matrices in $\mathbb{I}(A,B)$ are $TN_k$ if and only if the
matrices $\{ I_{z,z'}(A,B) : z \in \{ \pm 1 \}^m, z' \in \{ \pm 1 \}^n
\}$ are all $TN_k$.
\end{utheorem}

As in Theorem~\ref{tp_hull_k}, note that this test set is independent of
$k$.


\begin{rem}
Note that the two matrices $C^\pm(A,B)$ in Theorem~\ref{tp_hull_k} do not
always suffice as test matrices for the interval hull $\mathbb{I}(A,B)$
to be $TN_k$ as in Theorem~\ref{Ttnk}. For instance, let $n \geq 4$ and
$m, k \geq 3$, and define
\[
A_{m \times n} = \begin{pmatrix} A' & {\bf 0} \\ {\bf 0} & {\bf 0}
\end{pmatrix}, \ B_{m \times n} = \begin{pmatrix} B' &
{\bf 0} \\ {\bf 0} & {\bf 0} \end{pmatrix}, \quad \text{where} \quad
A' := \begin{pmatrix} 3 & 1 & 0 & 1 \\ 2 & 2 & 0 & 2 \\ 1 & 1 & 0 & 1
\end{pmatrix}, \ B' := \begin{pmatrix} 4 & 2 & 0 & 2 \\ 3 & 2 & 0 & 2 \\
1 & 1 & 0 & 1
\end{pmatrix}.
\]
It is easily verified that $C^\pm(A,B)$ are both $TN$, and the matrix
$C_{m \times n} = \begin{pmatrix} C' & {\bf 0} \\ {\bf 0} & {\bf 0}
\end{pmatrix}$ lies in $\mathbb{I}(A,B)$, where
\[
C' := \begin{pmatrix}
3 & 2 & 0 & 1 \\ 3 & 2 & 0 & 2 \\ 1 & 1 & 0 & 1
\end{pmatrix},
\]
yet $C$ has a negative $3 \times 3$ minor.
\end{rem}

We conclude the discussion of our main results by returning to the
simplicity of the sign non-reversal condition in Theorems
\ref{tp-sign-rev_k} and \ref{ThmC} to characterize total
positivity/non-negativity. The equivalence of $TP_k$ or $TN_k$ to all
square submatrices satisfying a property is reminiscent of the
well-known, fundamental characterization of total positivity by
Gantmacher--Krein~\cite{gantmacher-krein} (or of $P$-matrices
in~\cite{gale-nikai-pmat}), which we now recall for the reader's
convenience:

\begin{theorem}\label{Tgk}
Given a rectangular matrix $A \in \mathbb{R}^{m \times n}$, the following
statements are equivalent.
\begin{enumerate}
	\item The matrix $A$ is totally positive of order $k$.
	\item Every square submatrix of $A$ of size $r \leq k$ has
	positive (and simple) eigenvalues.
\end{enumerate}
\end{theorem}

In a sense, assertions $(1) \Longleftrightarrow (2)$ in
Theorems~\ref{tp-sign-rev_k} and~\ref{ThmC} resemble this result in the
structure of the second assertions and the simplicity of the statements.
Similarly, recall the well known paper by Fomin and Zelevinsky
\cite{FZ00} about tests for totally positive matrices, as well as recent
follow-ups such as~\cite{Chepuri}. Our results may be regarded as being
similar in spirit.

\section{The proofs}

\subsection{The sign non-reversal characterization of total positivity}

We begin by proving Theorem~\ref{tp-sign-rev_k}; this requires two
preliminary results. The first, from 1965, establishes a sign
non-reversal phenomenon for matrices with positive principal minors
(these are known as $P$-matrices):

\begin{theorem}[Gale--Nikaido,~\cite{gale-nikai-pmat}]\label{snrp}
A matrix $A \in \mathbb{R}^{n \times n}$ has all principal minors
positive if and only if for all $0 \neq x \in \mathbb{R}^n$, there exists
$i \in \langle n \rangle$ such that $x_i (Ax)_i > 0$.
\end{theorem}

The next preliminary result is the well-known 1912 result of Fekete for
$TP$ matrices, subsequently extended in 1955 by Schoenberg to $TP_k$
matrices.

\begin{theorem}[Fekete~\cite{FP12}, Schoenberg~\cite{S55}]\label{fec}
Let $m,n \geq k \geq 1$ be integers. Then $A \in \mathbb{R}^{m \times n}$
is $TP_k$ if and only if all contiguous submatrices of $A$ of size at
most $k$ have positive determinant.
\end{theorem}

See also \cite[pp.~78]{fallat-john}. For more details about totally
positive matrices, we refer to \cite{K68,pinkus}.

\begin{proof}[Proof of Theorem \ref{tp-sign-rev_k}]
    That $(1) \implies (2)$ follows from Theorem \ref{snrp}, while $(2)
    \implies (3) \implies (4)$ is immediate. To show $(4) \implies (1)$,
    by Theorem~\ref{fec} it suffices to show all contiguous $r \times r$
    minors of $A$ are positive, for $r \leq k$. We prove this by
    induction on $r$, with the $r=1$ case immediate from (4)(i),(ii). Now
    suppose all contiguous minors of $A$ of size at most $(r-1)$ are
    positive, and $B$ is a contiguous $r \times r$ submatrix of $A$. Then
    all proper minors of $B$ are positive, by Theorem~\ref{fec}.
    
    We first claim that $B$ is invertible. Indeed, suppose $Bx = 0$, so
    that $x \not\in \altr{r}$. Moreover, all components of $x$ are
    nonzero, else a proper minor of $B$ vanishes. Now partition $x \neq
    0$ into `contiguous coordinates of like signs':
    \[
    (x_1,\ldots ,x_{s_1}), \quad (x_{s_1 +1},\ldots ,x_{s_2}), \quad
    \ldots \quad (x_{s_u+1},\ldots ,x_r),
    \]
    with all coordinates in the $i$th component having the same sign,
    which equals $(-1)^{i-1}$ without loss of generality. Set $s_0:=0$
    and $s_{u+1}:=r$, and note that $u \leq r-2$ since $x \not\in
    \altr{r}$. Let $c^1,\ldots, c^r \in \mathbb{R}^r$ denote the columns
    of $B$, and define
    \[
    y^j:= \sum\limits_{l=s_{j-1}+1}^{s_j}\vert x_l \vert c^l, \qquad j
    \in \langle u+1 \rangle.
    \]
    Let $Y:=[y^1,y^2,\ldots, y^{u+1}] \in \mathbb{R}^{r \times (u+1)}$.
    We claim $Y$ is totally positive. Since no $x_j$ is zero, and all
    proper minors of $B$ are positive, for an integer $p \leq u+1$ and
    $p$-element subsets $I \subset \langle r \rangle$, $J\subseteq
    \langle u+1 \rangle$ one computes using standard properties of
    determinants:
    \[
    \det Y_{I\times J}= \sum_{l_1=s_{j_1-1}+1}^{s_{j_1}} \cdots
    \sum_{l_p=s_{j_p-1}+1}^{s_{j_p}}  \vert x_{l_1} \vert \ldots \vert
    x_{l_p} \vert \det B_{I\times L} > 0,
    \]
    where $L=\{l_1,\ldots l_p\}$, and $Y_{I \times J}$ denotes the
    submatrix of $Y$ with rows and columns indexed by $I,J$ respectively;
    a similar notation applies to $B_{I \times L}$.
    Hence $Y$ is totally positive.  But $Y \bd{u+1} = Bx = 0$, a
    contradiction. Thus $B$ is invertible, as claimed.

    Finally, we claim $\det B > 0 $. Define $z^B$ as in~\eqref{Ezb}; then
    \begin{equation}\label{tpalt}
    z^B_i= (-1)^{i-1} (\det B) \sum_{j=1}^r \det B_{(\langle r
    \rangle\setminus\{j\}) \times (\langle r \rangle \setminus\{i\})},
    \qquad i \in \langle r \rangle
    \end{equation}
    and the summation on the right is positive by assumption. Thus $z^B
    \in \altr{n}$, so by~(4)(ii) and~\eqref{tpalt}, there exists $i \in
    \langle r \rangle$ such that
    \begin{equation}\label{Ecramer2}
    0 \leq z^B_i (Bz^B)_i = z^B_i (\det B)^2 (-1)^{i-1} = (\det B)^3
    \sum_{j=1}^r \det B_{(\langle r \rangle\setminus\{j\}) \times
    (\langle r \rangle \setminus\{i\})}.
    \end{equation}
    From this it follows that $\det B > 0$, and the induction step is
    complete.
\end{proof}

\subsection{A test for total positivity (of any order) of the interval
hull of rectangular matrices}

We next demonstrate the aforementioned test for total positivity of the
interval hull $\mathbb{I}(A,B)$, as in Theorem~\ref{tp_hull_k}. We first
mention two preliminary results that are used in the proof. These require
the following notation.

\begin{defn}
Fix integers $m,n \geq 1$ and matrices $A, B \in \mathbb{R}^{m \times
n}$, with interval hull $\mathbb{I}(A,B)$.
\begin{enumerate}
	\item Define the $m \times n$ matrices $I_u, I_l$ via:
	$(I_u)_{ij}:= \max\{a_{ij},b_{ij}\}, (I_l)_{ij}:=
	\min\{a_{ij},b_{ij}\}$.
	\item For matrices (or vectors) $A,B$, write $A \leq B$ if all
	entries of $B-A$ are non-negative.
\end{enumerate}
\end{defn}

The first preliminary result is a straightforward verification.

\begin{lemma}\label{int-lem1}
    Let $m, n \geq 1$ and $A,B \in \mathbb{R}^{m \times n}$. Then $I_u,
    I_l, C^\pm(A,B) \in \mathbb{I}(A,B)$. If $m=n$, then $I_{z,z}(A,B)
    \in \mathbb{I}(A,B)$ for all $z \in \{ \pm 1 \}^n$.
\end{lemma}

The next lemma is precisely \cite[Theorem 2.1]{RJRG}; as the proof is
short, we include it.

\begin{lemma}\label{rohn_exten2}
    Fix $n \geq 1$ and $A,B \in \mathbb{R}^{n \times n}$ and $x \in
    \mathbb{R}^n$. Let the tuple $z \in \{ \pm 1 \}^n$ be such that $z_i
    = 1$ if $x_i \geq 0$ and $z_i = -1$ if $x_i < 0$. If $C \in
    \mathbb{I}(A,B)$, then
    \[
    x_i(Cx)_i\geq x_i (I_{z,z}(A,B) x)_i, \quad \forall i \in \langle n
    \rangle.
    \]
\end{lemma}

\begin{proof}
    In this proof we write $I_c := (A+B)/2$ and $\Delta := |A-B|/2$ for
    ease of exposition. Given $C \in \mathbb{I}(A,B) = \mathbb{I}(I_l,
    I_u)$,
    \[
    I_c - \Delta = I_l \leq C \leq I_u = I_c + \Delta.
    \]
    From this -- and given $0 \neq x \in \mathbb{R}^n$ and
    $i \in \langle n \rangle$ -- it follows via the triangle inequality
    that
    \[
    |x_i((C-I_c)x)_i | \leq | x_i | (| C-I_c| \cdot | x | )_i
    \leq |x_i| ( \Delta  |x|)_i.
    \]
    From this, we compute using that $|x| = D_z x$:
    \[
    x_i (Cx)_i \geq x_i (I_c x)_i - z_i x_i (\Delta D_z x)_i = x_i (I_c
    x)_i - x_i (D_z \Delta D_z x)_i = x_i (I_{z,z}(A,B) x)_i. \qedhere
    \]
\end{proof}

With these preliminaries at hand, we have:

\begin{proof}[Proof of Theorem \ref{tp_hull_k}]
    If $\mathbb{I}(A, B)$ is $TP_k$, then so are $C^\pm(A,B)$ by Lemma
    \ref{int-lem1}.
    Conversely, suppose $C^\pm(A,B)$ are $TP_k$, and let $C \in
    \mathbb{I}(A,B)$. Fix $r \in \langle k \rangle$ and a vector $x \in
    \altr{r}$. Now let $C'$ be a $r \times r$ contiguous submatrix of
    $C$, say $C' = C_{J \times K}$ for contiguous sets of indices $J, K
    \subseteq \langle k \rangle$ with $|J| = |K| = r$. It suffices to
    show by Theorem~\ref{tp-sign-rev_k}(3) that $x_i (C'x)_i > 0$ for
    some $i \in \langle{r} \rangle$.

    To proceed, we require some notation. Let $A', B'$ be contiguous
    submatrices of $A,B$ respectively, consisting of the entries in the
    same positions as $C'$. Now since $C' \in \mathbb{I}(A',B')$,
    Lemma~\ref{rohn_exten2} implies for some $i \in \langle r \rangle$
    that
    \begin{equation}\label{Erohn}
    x_i (C' x)_i \geq x_i (I_{\z{r}, \z{r}}(A',B')x)_i,
    \end{equation}
    where the vector $\z{r} \in \{ \pm 1 \}^r \cap \altr{r}$ is given by
    $\z{r}_j := x_j / |x_j|\ \forall j$. Since $C' = C_{J \times K}$, we
    can embed the vector of signs $\z{r}$ in contiguous positions $J
    \subseteq \langle m \rangle$ and $K \subseteq \langle n \rangle$, and
    uniquely extend to alternating $\pm 1$-valued vectors of lengths
    $m,n$ respectively. Formally, there exist unique signs
    $\varepsilon_J, \varepsilon_K \in \{ \pm 1 \}$, such that $\z{r}$ is
    the restriction to positions $J$ (respectively $K$) of $\varepsilon_J
    \bd{m} \in \altr{m}$ (respectively $\varepsilon_K \bd{n} \in
    \altr{n}$). But then $I_{\z{r},\z{r}}(A',B')$ is a contiguous $r
    \times r$ submatrix of
    \[
    I_{\varepsilon_J \bd{m},\; \varepsilon_K \bd{n}}(A,B) = \frac{A+B}{2}
    - \varepsilon_J \varepsilon_K D_{\bd{m}} \frac{|A-B|}{2} D_{\bd{n}}
    \in \{ C^+(A,B), C^-(A,B) \}.
    \]
    By assumption, this matrix is $TP_k$, so $I_{\z{r},\z{r}}(A',B')$ is
    $TP$. Using Theorem~\ref{tp-sign-rev_k}(3) and~\eqref{Erohn}, we have
    \[
    x_i (C' x)_i \geq x_i (I_{\z{r},\z{r}}(A',B')x)_i > 0.
    \]
    Thus $C'$ has the sign non-reversal property with respect to all $x
    \in \altr{r}$. By Theorem~\ref{tp-sign-rev_k}, $C$ is $TP_k$.
\end{proof}

\begin{rem}
We remark that Garloff's results for $n \times n$ matrices in
\cite{Gar82} (see also \cite[Chapter~3]{pinkus}) are stated with respect
to the checkerboard ordering $\leq^*$, in which
\[
A \leq ^* B \quad \Longleftrightarrow \quad D_{\bd{n}} (B - A) D_{\bd{n}}
\geq 0_{n \times n}.
\]
Now an easy computation shows that the hull $\mathbb{I}(A,B)$ itself does
not depend on whether one uses the entrywise ordering or the checkerboard
ordering. More precisely,
\[
\mathbb{I}(A,B) = \{ C : I_l \leq C \leq I_u \} = \{ C : C^-(A,B) \leq^*
C \leq^* C^+(A,B) \}.
\]
In particular, the test set of $\{ C^+(A,B), C^-(A,B) \}$ works
regardless of which ordering is used -- and works for all rectangular
matrices and for testing the $TP_k$ property for any $k \geq 1$.
\end{rem}

\begin{rem}\label{Rgarloff}
Garloff has pointed out to us that Theorem~\ref{tp_hull_k} can also be
proved by reducing to the case of square $TP$ matrices (which is his
result in~\cite{Gar82}) as follows: if $C^\pm(A,B) \in \mathbb{R}^{m
\times n}$ are both $TP_k$, then every contiguous square submatrix of
$C^\pm(A,B)$ of size $r \in \langle k \rangle$ is $TP$. Now given $C \in
\mathbb{I}(A,B)$, every contiguous square submatrix of $C$ of size $r \in
\langle k \rangle$ lies in between square $TP$ submatrices of $C^+(A,B)$
and $C^-(A,B)$, hence is $TP$ by~\cite{Gar82}. Now use the
Fekete--Schoenberg Theorem~\ref{fec}.
\end{rem}

\subsection{The sign non-reversal characterization and interval test of
total non-negativity}\label{STNk}

We conclude by proving our results involving the $TN_k$ property.
Theorem~\ref{ThmC}, which characterizes the $TN_k$ property, requires a
classical density result by Whitney:

\begin{theorem}[Whitney,~\cite{Whitney}]\label{Twhitney}
Given integers $m,n \geq k \geq 1$, the set of $m \times n$ $TP_k$
matrices is dense in the set of $m \times n$ $TN_k$ matrices.
\end{theorem}

This helps to show the promised characterization:

\begin{proof}[Proof of Theorem~\ref{ThmC}]
We prove a cyclic chain of implications.
First suppose $A \in \mathbb{R}^{m \times n}$ is $TN_k$. By
Theorem~\ref{Twhitney}, there exists a sequence $A^{(l)} \to A$ of $TP_k$
matrices. Fix $r \in \langle k \rangle$ and an $r \times r$ submatrix $B$
of $A$, and let $B^{(l)}$ be the submatrix of $A^{(l)}$ in the same
positions as $B$. Also fix a vector $0 \neq x \in \mathbb{R}^r$, and let
$J \subseteq \langle r \rangle$ index the nonzero components of $x$. Now
$B^{(l)}$ is $TP$, so by Theorem~\ref{tp-sign-rev_k}(2) there exists $j_l
\in J$ such that $x_{j_l} (B^{(l)} x)_{j_l} > 0$. Hence there exists an
(increasing) subsequence $l_p, p \geq 1$ of positive integers such that
all $j_{l_p}$ equal the same entry in $J$, say $j_0$. But then,
\[
x_{j_0} (Bx)_{j_0} = \lim_{p \to \infty} x_{j_{l_p}} (B^{(l_p)}
x)_{j_{l_p}} \geq 0, \qquad x_{j_0} \neq 0.
\]

Thus $(1) \implies (2)$. Clearly $(2) \implies (3) \implies (4)$. Now
assume~(4); we show that $\det B \geq 0$ for all $r \times r$ submatrices
$B$ of $A$, by induction on $r \in \langle k \rangle$.
Indeed, the base case $r=1$ is immediate. For the induction step, if
$\det B = 0$ then we are done; else $\det(B) \neq 0$. Now we once again
have~\eqref{tpalt}, and the sum on the right is non-negative. But since
$B$ is invertible, no column of $\adj(B)$ is zero, so the sum on the
right is in fact positive by the induction hypothesis. This implies $z^B
\in \altr{r}$, so we may repeat the calculation in~\eqref{Ecramer2} and
deduce that $\det B > 0$. This proves~(1) by induction.
\end{proof}

We conclude by showing our final remaining result:

\begin{proof}[Proof of Theorem~\ref{Ttnk}]
First verify that $I_{z,z'}(A,B) \in \mathbb{I}(A,B)$ for all $z,z'$.
Now repeat the proof of Theorem~\ref{tp_hull_k}, but working with
arbitrary $r \times r$ submatrices $C'$ of $C \in \mathbb{I}(A,B)$, where
$r \in \langle k \rangle$. Once again obtaining~\eqref{Erohn}, we note
that $I_{\z{r},\z{r}}(A',B')$ is a submatrix of $I_{z,z'}(A,B)$ for some
$z,z'$, and the latter is $TN_k$ by assumption. The remainder of the
proof is unchanged, except for the use of Theorem~\ref{ThmC} instead of
Theorem~\ref{tp-sign-rev_k}.
\end{proof}

\section*{Acknowledgements}
We thank J\"{u}rgen Garloff for directing us to several prior results in
the literature on interval hulls of TP matrices, for observing his
alternate proof in Remark~\ref{Rgarloff}, and for several helpful
remarks on an earlier version of this work that helped improve the
exposition. We also thank the referee for providing useful comments and
references that improved the manuscript.

P.N.~Choudhury was supported by National Post-Doctoral Fellowship
(PDF/2019/000275) from SERB (Govt.~of India) and the NBHM Post-Doctoral
Fellowship (0204/11/2018/R$\&$D-II/6437) from DAE (Govt.~of India).
M.R.~Kannan would like to thank  the SERB, Department of Science and
Technology, India, for financial support through the projects MATRICS
(MTR/2018/000986) and Early Career Research Award (ECR/2017/000643).
A.~Khare was partially supported by
Ramanujan Fellowship grant SB/S2/RJN-121/2017,
MATRICS grant MTR/2017/000295, and
SwarnaJayanti Fellowship grants SB/SJF/2019-20/14 and DST/SJF/MS/2019/3
from SERB and DST (Govt.~of India),
and by grant F.510/25/CAS-II/2018(SAP-I) from UGC (Govt.~of India).

\end{document}